\documentclass[a4paper,reqno,11pt]{amsart}
\usepackage[utf8]{inputenc}
\usepackage[margin=0.8in]{geometry}
\usepackage[svgnames]{xcolor}
\usepackage[bookmarksnumbered=true]{hyperref} 
\hypersetup{
	colorlinks = true,
	linkcolor = blue,
	anchorcolor = blue,
	citecolor = teal,
	filecolor = blue,
	urlcolor = blue
}
\usepackage[english]{babel}
\usepackage{amsmath}
\usepackage{amsthm}
\usepackage{amssymb}
\usepackage{mathrsfs}
\usepackage{tikz-cd} 
\usepackage{bm}

\usepackage{hyperref}
\usepackage{comment}
\usepackage{caption}

\usepackage{cleveref}

\newtheorem{thmx}{Theorem}
\newtheorem{theorem}{Theorem}[section]

\newtheorem{corollary}[theorem]{Corollary}
\newtheorem{lemma}[theorem]{Lemma}

\theoremstyle{definition}

\newtheorem{definition}[theorem]{Definition}
\newtheorem{remark}[theorem]{Remark}

\newtheorem{problem}[theorem]{Problem}

\newcommand*{\LP}{\operatorname{LP}}
\newcommand*{\lp}{\ell p}

\newcommand*{\chen}[1]{\textcolor{cyan}{#1}}

\newcommand*{\alessio}[1]{\textcolor{violet}{#1}}

\title{The leading coefficient of Lascoux polynomials}

\author[A. Borz\`{i}]{Alessio Borz\`{i}}
\address{Mathematics Institute, University of Warwick, United Kingdom}
\email{Alessio.Borzi@warwick.ac.uk}

\author[X. Chen]{Xiangying Chen}
\address{Institute of Algebra and Geometry, Otto von Guericke Universit\"at Magdeburg, Magdeburg, Germany}
\email{xiangying.chen@ovgu.de}

\author[H. J. Motwani]{Harshit J. Motwani}
\address{Department of Mathematics: Algebra and Geometry, Ghent University, 9000 Gent, Belgium} 
\email{harshitjitendra.motwani@ugent.be}

\author[L. Venturello]{Lorenzo Venturello}
\address{Department of Mathematics, KTH Royal Institute of Technology, Stockholm, Sweden}
\email{lven@kth.se}

\author[M. Vodi\v{c}ka]{Martin Vodi\v{c}ka}
\address{Max Planck Institute for Mathematics in the Sciences, Leipzig, Germany}
\email{vodicka@mis.mpg.de}

\begin{document}
	
	\maketitle
	
	\begin{abstract}
		Lascoux polynomials have been recently introduced to prove polynomiality of the maximum-likelihood degree of linear concentration models. We find the leading coefficient of the Lascoux polynomials (type C) and their generalizations to the case of general matrices (type A) and skew symmetric matrices (type D). In particular, we determine the degrees of such polynomials. As an application, we find the degree of the polynomial $\delta(m,n,n-s)$ of the algebraic degree of semidefinite programming, and when $s=1$ we find its leading coefficient for types C, A and D.
	\end{abstract}
	
	{\hypersetup{linkcolor=black}
		%\tableofcontents
	}
	
	\section{Introduction}
	In statistics, a multivariate Gaussian distribution is an important family of parametric statistical models, whose parameters are given by a mean vector $\mu \in \mathbb{R}^n$ and covariance matrix $\Sigma$ which is positive definite. The inverse $\Sigma^{-1}$ is called the \emph{concentration matrix}. 
	The problems studied in this paper are motivated by \emph{linear concentration models}, introduced by Anderson \cite{Anderson1970estimate}. In these models, the concentration matrix $\Sigma^{-1}$ is assumed to be in a $d$-dimensional linear subspace $\mathcal{L}$ of symmetric matrices, in particular $\Sigma$ should belong to the set $\mathcal{L}^{-1}$ of the inverse matrices of $\mathcal{L}$.
	
	An important invariant that measures the complexity of a linear concentration model is the \emph{maximum likelihood degree} (ML-degree), which is the number of critical points of the rational score equations coming from generic data points. If the linear space $\mathcal{L}$ is generic, the ML-degree is the degree of the Zariski closure of $\mathcal{L}^{-1}$ (see  \cite[Theorem 1]{sturmfels2010multivariate} or \cite[Corollary 2.6]{michalek2021cstaractions}). In this case, the ML-degree depends just on the size $n$ of the symmetric matrices and the dimension $d$ of $\mathcal{L}$, and it will be denoted by $\phi(n,d)$.
	\begin{comment}
	Thus
	\begin{center}
	\emph{$\phi(n,d)$ is the degree of the Zariski closure of $\mathcal{L}^{-1}$, where $\mathcal{L}$ is a \\ $d$-dimensional linear subspace of $n \times n$ symmetric matrices.}
	\end{center}
	\end{comment}
	
	In \cite{michalek2021cstaractions} a new connection of the ML-degree with enumerative geometry was found. This allowed new techniques and tools to study the ML-degree. For instance, $\phi(n,d)$ can be defined in pure enumerative terms, as being the number of nondegenerate quadrics in $n$ variables, passing through $\binom{n+1}{2}-d$ general points and tangent to $d-1$ general hyperplanes. Such problems can be solved by performing computations in the cohomology ring of the variety of \emph{complete quadrics}. In light of this connection, later in \cite{manivel2020complete} the following polynomiality result, previously conjectured by Sturmfels and Uhler \cite[p. 611]{sturmfels2010multivariate} (see also \cite[Conjecture 2.8]{michalek2021cstaractions}) was settled:
	
	\begin{thmx}\cite[Theorem 1.3]{manivel2020complete}
		%For any fixed positive integer $d$, the ML-degree $\phi(n,d)$ is a polynomial in $n$.
		For any $d > 0$ fixed, the function $n \mapsto \phi(n,d)$ is polynomial.
	\end{thmx}
	The proof of the previous theorem boils down to show the polynomiality of certain functions \cite[Theorem 4.3]{manivel2020complete}, called \emph{Lascoux polynomials}, after Alain Lascoux \cite{lascoux1989giambelli}. There are several equivalent ways to define Lascoux polynomials. For instance, in \Cref{sec:C} we will give a definition in terms of Schur polynomials. Here we describe Lascoux polynomials in a more elementary manner. First, consider the infinite Pascal triangle matrix
	\[ E = \begin{pmatrix}
		1 & 0 & 0 & 0 & \dots \\
		1 & 1 & 0 & 0 & \dots \\
		1 & 2 & 1 & 0 & \dots \\
		1 & 3 & 3 & 1 & \dots \\
		\vdots & \vdots & \vdots & \vdots & \ddots
	\end{pmatrix} \]
	where $E_{ij} = \binom{i}{j}$. 
	For every pair of finite subsets $I,J \subseteq \mathbb{N}$, let $E_{I,J}$ be the submatrix of $E$ with rows indexed by $I$ and columns indexed by $J$. The \emph{Lascoux coefficient} $\psi_I$ of a finite subset $I \subseteq \mathbb{N}$ of cardinality $r$, is defined by
	\[ \psi_I = \sum_{J \subseteq \mathbb{N}, \, |J| = r} \det (E_{I,J}). \]
	Observe that the sum above has only finitely many non-zero terms.
	For every nonnegative integer $n \geq 0$, let $[n] = \{ 0,1,\dots,n-1 \}$. The Lascoux polynomial of a finite subset $I \subseteq \mathbb{N}$, is the function
	\[ \LP_I(n) = \begin{cases}
		\psi_{[n] \setminus I} & \text{if } I \subseteq [n], \\
		0 & \text{otherwise.}
	\end{cases} \]
	
	Two proofs of the polynomiality of $\LP_I$ were provided in \cite{manivel2020complete}. The first uses two recursive formulas of the Lascoux polynomials. In the second, the authors dive in to the properties of the minors of the Pascal triangle matrix. Although the techniques used in the second proof are longer and more technical, they allow to find the degree and leading coefficient of the Lascoux polynomials:
	
	\begin{thmx}\cite[Theorem 4.12]{manivel2020complete} \normalfont{(Type C, \Cref{thm:leading-coefficient-C})}
		Let $I = \{ i_1 < \dots < i_r \} \subseteq \mathbb{N}$. The polynomial $\LP_I$ has degree $\sum_k i_k + |I|$. Its leading coefficient is equal to
		\[
		\frac{\prod_{j>k}(i_j-i_k)}{(i_1+1)!\cdots(i_r+1)!\prod_{j>k}(i_j+i_k+2)}.
		\]
	\end{thmx}
	
	Our first main contribution is to provide a more direct proof of the previous theorem, starting from the recurrence relations of the Lascoux polynomials.
	
	Further, all the results mentioned above have natural analogues if we replace the space of symmetric matrices (``type C", \Cref{sec:C}) with the space of general matrices (``type A", see \Cref{sec:A}) or with the space of skew-symmetric matrices (``type D", \Cref{sec:D}). This point of view was already taken in \cite{manivel2020complete}. We continue in this direction, finding formulas for the degree and leading coefficient of Lascoux polynomials for type A and D, which were not previously known.
	\begin{thmx}[Type A, \Cref{thm: leadingcoefA}]
		For sets $I=\{i_1,...,i_r\}$, $J=\{j_1,...,j_r\}$, 
		the degree of the Lascoux polynomials of type A is given by the following expression on $I, J$:
		\[
		\deg(\LP^A_{I,J}(n)) = |I| + \sum I + \sum J
		\]
		and the leading coefficient of $\LP^A_{I,J}$ is
		\[\frac{\prod_{k>l}(i_k-i_l)\prod_{k > l} (j_k-j_l)}{\prod_{k,l=1}^r(i_k+j_l+1) \prod_{k=1}^r (i_k)! \prod_{k=1}^r (j_k)!}.\]
	\end{thmx}
	\begin{thmx}[Type D, \Cref{thm: leadingcoefD}]
		Let $I=\{i_1<\cdots <i_r\}\subset \mathbb{N}$ be a set of nonnegative integers. Then
		\begin{itemize}
			\item[-] If $i_1>0$, $\LP^D_{I}(2n)$ and $\LP^D_{I}(2n+1)$ are polynomials in $n$ of degree $\sum I$ and leading coefficient equal to
			\[
			\frac{2^{\sum I-|I|}\prod_{k>l}(i_k-i_l)}{\prod_{k>l}(i_k+i_l) \prod_k (i_k)!}.
			\]
			\item[-] If $i_1=0$, then $\LP^D_{I}(n)=\LP^D_{I\setminus\{0\}}(n)$ if $n-|I|$ is even, and  $\LP^D_{I}(n)=0$ if $n-|I|$ is odd.  
		\end{itemize}
	\end{thmx}
	Lascoux coefficients also appear in the context of \emph{semidefinite programming} (SDP), a subject in optimization theory that concerns the problem of optimizing a linear function over the cone of positive semidefinite matrices. An important invariant that addresses the complexity of these problems is the \emph{algebraic degree of semidefinite programming}. For more information about the algebraic degree of SDP, we refer to \cite{NRS}. Following \cite[Definition 1.4]{manivel2020complete}, here we provide the following definition of the algebraic degree of SDP in the language of algebraic geometry. Let $\mathcal{L} \subseteq S^2 \mathbb{C}^n$ be a general linear space of symmetric matrices of affine dimension $m+1$, and let $SD^{r,n}_m \subseteq \mathbb{P}(\mathcal{L})$ denote the projectivization of the cone of matrices of rank at most $r$ in $\mathcal{L}$. The algebraic degree of SDP, denoted $\delta(m,n,r)$, is the degree of the projective dual of $SD^{r,n}_m$ if this dual is a hypersurface, and zero otherwise.
	
	In \cite{BR} the authors found a formula that expresses $\delta(m,n,r)$ in terms of Lascoux coefficients. In addition, in \cite{manivel2020complete} the authors proved that the function $n \mapsto \delta(m,n,n-s)$ for fixed $m,s > 0$ is a polynomial, and provided another formula for $\delta(m,n,n-s)$ previously conjectured in \cite[Conjeture 21]{NRS}. Similarly for Lascoux polynomials, the results in \cite{manivel2020complete} were also proved for type A and D (see \Cref{sec:delta} for the related definitions). As an application of our previous results for the Lascoux polynomials, we find the degree of the polynomials $\delta(m,n,n-s)$, and their leading coefficient for $s=1$, in type C, A and D.
	
	\begin{thmx} Let $s>0$.
		\begin{itemize}
			\item[-]\normalfont{(Type C, \Cref{thm: semidefinite type C})} The polynomial $\delta(m,n,n-s)$ has degree $m$, for every $m\geq \binom{s+1}{2}$. Moreover
			\[
			LC(\delta(m,n,n-1)) = \frac{2^{m-1}}{m!},
			\]			
			for every $m>0$.
			\item[-]\normalfont{(Type A, \Cref{thm: semidefinite type A})} The polynomial $\delta_A(m,n,n-s)$ has degree $m$, for every $m\geq s^2$. Moreover,
			\[
			LC(\delta_A(m,n,n-1)) = \frac{1}{m!}\binom{2(m-1)}{m-1},
			\]			
			for every $m>0$.
			\item[-]\normalfont{(Type D, \Cref{thm: semidefinite type D})} The polynomial $\delta_D(m,n,n-s)$ has degree $m$, for every $m\geq \binom{2s}{2}$. Moreover, 
			\[
			LC(\delta_D(m,n,n-1)) =\frac{2^{m-2}}{m!}\left(\frac{1}{m}\binom{2(m-1)}{m-1}+1\right),
			\]		
			for every $m>0$.	
		\end{itemize}
		
	\end{thmx}

	This paper is organized as follows.  %In Section \ref{sec:preliminary} we recall the basics of Schur polynomials and partitions.
	In Section \ref{sec:technical} we prove some technical lemmas that will be used throughout the paper, in Section \ref{sec:C}, \ref{sec:A} and \ref{sec:D} we find the degree and the leading coefficient of the Lascoux polynomials for type C, A and D respectively. Finally, in Section \ref{sec:delta} we find the algebraic degrees of $\delta(m,n,n-1)$ for type C, A and D.
	
	%specify Segre class
	%change |\lambda| with \sum \lambda
	\begin{remark}
		We would like to point out that the terminology ``Lascoux polynomials" appears in the literature in more than one context not necessarily related to our setting. Our choice is motivated by the definitions in \cite{manivel2020complete}. 
	\end{remark}
	~\\
	\textbf{Acknowledgements.} This project originated during the online workshop \href{https://sites.google.com/view/react-2021}{REACT}. The
	authors would like to express their gratitude to the organizers and to the lecturers.  The authors are especially grateful to the lecturers Mateusz Michałek and Tim Seynnaeve for suggesting this topic and for several helpful discussions. H.J. Motwani was partially supported by UGent BOF/STA/201909/038 and FWO grant G0F5921N. L. Venturello was funded by the G\"{o}ran Gustafsson foundation.
	
	\begin{comment}
	
	\section{Preliminaries on Schur polynomials}\label{sec:preliminary}
	
	A \emph{partition} $\lambda$ is a nonincreasing sequence of nonnegative integers $(\lambda_1,\dots,\lambda_r)$. The \emph{length} of the partition is the length of the sequence, the \emph{weight} is $\sum\lambda = \sum_{i=1}^r \lambda_i$.
	For a set $I = \{ i_1,\dots,i_r \}$ of nonnegative integers with $i_1 < i_2 < \dots < i_r$, we define the corresponding partition
	\[ \lambda(I) = ( i_r-(r-1), i_{r-1}-(r-2),\dots,i_2-1,i_1 ). \]
	For a partition $\lambda$ of \chen{length} $k$ its associated \emph{Schur polynomial} $s_\lambda$ is defined as follows:
	\[ s_\lambda(x_1,\dots,x_k) = \frac{\det (x_j^{\lambda_i + k - i})_{ij}}{\det (x_j^{k - i})_{ij}}. \]
	Note that the denominator of $s_\lambda$ is the Vandermonde determinant $\prod_{i < j}(x_i-x_j)$. The degree of $s_\lambda$ is equal to the weight $\sum\lambda$ of the partition. As an example, the elementary symmetric polynomial in $k$ variables of degree $r$ is the Schur polynomial with partition $\lambda = (\underbrace{1,\dots,1}_{r},0,\dots,0)$ of length $k$: 
	\[ s_{\lambda}(x_1,\dots,x_k) = \sum_{i_1 < \dots < i_r} x_{i_1} \dots x_{i_r}. \]
	
	Throughout the paper, the leading coefficient of a polynomial $p$ will be denoted by $LC(p)$.
	
	\end{comment}
	
	\section{Four identities}\label{sec:technical}
	
	In this paper we will need the following four identities of rational functions. All the identities are thought to be in $k(x_1,\dots,x_r,y_1,\dots,y_r)$, where $k$ is a field of characteristic zero. We start with a ``Double Sum Lemma", expressing the sum of two sets of $r$ variables as a certain sum of rational functions.\\
	\begin{lemma}[Double Sum Lemma]\label{lem:doublesumlemma}
		The identity
		\begin{align}\label{doublesumlemma}
			\sum_{i=1}^r x_i + \sum_{j=1}^r y_j + r=\sum^r_{t=1} x_t \prod_{k \neq t}\frac{x_k - x_t +1}{x_k-x_t} \prod_{l=1}^r\frac{x_t+y_l+1}{x_t+y_l} + \sum^r_{m=1} y_m \prod_{k \neq m}\frac{y_k - y_m +1}{y_k-y_m}\prod_{l=1}^r\frac{x_l+y_m+1}{x_l+y_m}
		\end{align}
		holds for every $r\geq 1$.
	\end{lemma}
	\begin{proof}
		We write the right-hand side of (\ref{doublesumlemma}) with a common denominator
		\begin{equation}\label{eq: double sum lemma}
			\frac{\prod_{k>l}(y_k-y_l)A(x_1,\dots,x_r,y_1,\dots,y_r)+\prod_{k>l}(x_k-x_l)B(x_1,\dots,x_r,y_1,\dots,y_r) } {\prod_{k>l}(x_k-x_l)(y_k-y_l)\prod_{k,l=1}^r (x_k+y_l)},
		\end{equation}
		with 
		\[
		A(x_1,\dots,x_r,y_1,\dots,y_r)= \sum_{t=1}^r (-1)^{t-1}x_t\prod_{\substack{k>l\\ k,l\neq t}}(x_k-x_l)\prod_{k\neq t}(x_k-x_t+1)\prod_{l=1}^r(x_t+y_l+1)\prod_{\substack{k,l=1\\k\neq t}}^r(x_k+y_l)
		\]
		and
		\[
		B(x_1,\dots,x_r,y_1,\dots,y_r)=\sum_{m=1}^r (-1)^{m-1}y_m\prod_{\substack{k>l\\ k,l\neq m}}(y_k-y_l)\prod_{k\neq m}(y_k-y_m+1)\prod_{l=1}^r(x_l+y_m+1)\prod_{\substack{k,l=1\\ l\neq m}}^r(x_k+y_l).
		\]
		\textbf{Claim 1:} If we swap the role of $x_a$ and $x_b$, for some $1\leq a<b\leq r$, then 
		\[
		A(x_1,\dots,x_b,\dots,x_a,\dots,x_r,y_1,\dots,y_r)=-A(x_1,\dots,x_r,y_1,\dots,y_r).
		\]
		We analyze each summand in $A(x_1,\dots,x_r,y_1,\dots,y_r)$ separately. If $t\notin\{a,b\}$, then the only factor in the $t$-th summand which is affected by the swap is $\prod_{\substack{k>l\\ k,l\neq t}}(x_k-x_l)$. More precisely, the linear forms $(x_k-x_a)$, with $a<k\leq b$ and the linear forms $(x_b-x_l)$, with $a<l<b$ change sign. As there are $2(b-a)-1$ many such factors, there is a change of sign in $\prod_{\substack{k>l\\ k,l\neq t}}(x_k-x_l)$. If $t=a$, then the only changes of sign are given by the linear forms $(x_b-x_k)$ with $a<k<b$, as each becomes $-(x_k-x_a)$. This accounts for a factor of $(-1)^{b-a-1}$. Together with $(-1)^{t-1}=(-1)^{a-1}$ we obtain $-(-1)^{b-1}$. Hence the $a$-th summand of $A(x_1,\dots,x_b,\dots,x_a,\dots,x_r,y_1,\dots,y_r)$ is equal to the $b$-th summand of $A(x_1,\dots,x_r,y_1,\dots,y_r)$, with the sign changed. The case $t=b$ is analogous. We then have that 
		\[
		A(x_1,\dots,x_r,y_1,\dots,y_r)=\prod_{k>l}(x_k-x_l)A'(x_1,\dots,x_r,y_1,\dots,y_r),
		\]
		for some polynomial $A'(x_1,\dots,x_r,y_1,\dots,y_r)$ which is invariant under the transposition of any two $x$-variables. In the same way we can show that 
		\[
		B(x_1,\dots,x_r,y_1,\dots,y_r)=\prod_{k>l}(y_k-y_l)B'(x_1,\dots,x_r,y_1,\dots,y_r),
		\]
		with $B'(x_1,\dots,x_r,y_1,\dots,y_r)$ invariant under the transposition of any two $y$-variables.\\
		\textbf{Claim 2:} The evaluation of the numerator of \eqref{eq: double sum lemma} in $x_a=-y_b$ is equal to $0$, for every $1\leq a,b\leq r$.
		Let us fix $a$ and $b$. Observe that $(x_a+y_b)$ is a factor in all summands of $A(x_1,\dots,x_r,y_1,\dots,y_r)$ with $t\neq a$, and it is a factor in all summands of $B(x_1,\dots,x_r,y_1,\dots,y_r)$ with $m\neq b$. We then have that, with $x_a=-y_b$, the numerator of \eqref{eq: double sum lemma} equals to:
		\begin{align*}
			&\prod_{k>l}(y_k-y_l)\left( (-1)^{a}\bm{y_b} \prod_{\substack{k>l\\ k,l\neq a}}(x_k-x_l)\prod_{k\neq a}\bm{(x_k+y_b+1)\prod_{l\neq b}(y_l-y_b+1)\prod_{\substack{k,l=1\\k\neq a\\ l\neq b}}^r(x_k+y_l)} \prod_{\substack{k=1\\k\neq a}}^r(x_k-x_a) \right) +\\
			& \prod_{k>l}(x_k-x_l)\left((-1)^{b-1}\bm{y_b}\prod_{\substack{k>l\\ k,l\neq b}}(y_k-y_l)\prod_{k\neq b}\bm{(y_k-y_b+1)\prod_{l\neq a}(x_l+y_b+1)\prod_{\substack{k,l=1\\ l\neq b\\ k\neq a}}^r(x_k+y_l)} \prod_{\substack{l=1\\ l\neq b}}^r(y_l-y_b)\right).
		\end{align*}
		Here we have highlighted in bold the factors which are common to the two summands. To conclude the proof of claim 2 we observe that
		\[
		\prod_{k>l}(x_k-x_l) = (-1)^{a-1}\prod_{\substack{k>l\\ k,l\neq a}}(x_k-x_l)\prod_{\substack{k=1\\k\neq a}}^r(x_k-x_a)
		\]	
		and 
		\[
		\prod_{k>l}(y_k-y_l) = (-1)^{b-1}\prod_{\substack{k>l\\ k,l\neq b}}(y_k-y_l)\prod_{\substack{l=1\\ l\neq b}}^r(y_l-y_b).
		\]	
		This implies that the two summands above contain precisely the same factors in absolute value. As the first is multiplied by $(-1)^{2a-1}=-1$ and the second is multiplied by $(-1)^{2b-2}=1$, those cancel out.\\
		We conclude that the numerator of \eqref{eq: double sum lemma} equals 
		\[
		\prod_{k>l}(x_k-x_l)\prod_{k>l}(y_k-y_l)\prod_{k,l=1}^r (x_k+y_l)Q(x_1,\dots,x_r,y_1,\dots,y_r),
		\] 
		for some polynomial $Q(x_1,\dots,x_r,y_1,\dots,y_r)$ invariant under the transposition of any two $x$-variables and any two $y$-variables. A simple counting of the factors in \eqref{eq: double sum lemma} shows that the degree of $Q(x_1,\dots,x_r,y_1,\dots,y_r)$ is at most $1$. The vector space of polynomials of degree at most $1$ with this symmetry is $3$-dimensional, and therefore we can write 
		\[
		Q(x_1,\dots,x_r,y_1,\dots,y_r) = \lambda \sum_{i=1}^{r} x_i + \mu \sum_{j=1}^{r} y_j + \nu,
		\]
		for some $\lambda,\mu,\nu\in\mathbb{R}$. We first show that $\lambda=\mu=1$. If we order the variables as $y_r>\dots>y_1>x_r>\dots >x_1$, we obtain that the leading term of the numerator of \eqref{eq: double sum lemma} is $y_1^r y_2^{r+1}\cdots y_r^{2r-1} x_2 x_3^{2}\cdots x_{r-1}^{r-2}x_r^r$, while the leading term of the denominator equals to $y_1^r y_2^{r+1}\cdots y_r^{2r-1}x_2 x_3^2\cdots x_r^{r-1}$. The ratio of their coefficients, which is clearly equal to $1$, is the coefficient of $x_r$ in $Q(x_1,\dots,x_r,y_1,\dots,y_r)$, namely $\lambda$. If we order the variables as $x_r>\dots>x_1>y_r>\dots >y_1$
		we obtain in the same way that $\mu=1$.\\
		Finally, to conclude that $\nu=r$ we substitute $y_k=-x_k-1$ for every $1\leq k\leq r$ in the RHS of \eqref{doublesumlemma}. It is immediate to see that it vanishes, as both summands have $(x_t+y_t+1)$ as a factor, for some $t$. We deduce that $\sum_{i=1}^r x_i + \sum_{j=1}^r (-x_j-1) + \nu = -r + \nu =0$, and hence $\nu=r$.
	\end{proof}
	As a corollary, we obtain the following ``Sum Lemma".
	\begin{corollary}[Sum Lemma]\label{cor:sumlemma}
		For all positive integers $r$ the following identity holds:
		\begin{equation}\label{sumlemma}
			x_1+\dots+x_r=\sum_{l=1}^r x_l \prod_{j\neq l}  \frac{(x_j-x_l+1)(x_j+x_l)}{(x_j-x_l)(x_j+x_l-1)}.
		\end{equation}
	\end{corollary}
	\begin{proof}
		On substituting $y_l = x_l - 1$ for all $l$ in (\ref{doublesumlemma}) we get
		\begin{align*}
			\sum_{i=1}^r 2 x_i=& \sum^r_{t=1} x_t \prod_{k \neq t}\frac{x_k - x_t +1}{x_k-x_t} \prod_{l=1}^r\frac{x_t+x_l}{x_t+x_l-1} + \sum^r_{m=1}( x_m -1 )\prod_{k \neq m}\frac{x_k - x_m +1}{x_k-x_m}\prod_{l=1}^r\frac{x_l+x_m}{x_l+x_m-1}\\
			=&\sum^r_{t=1} \frac{2x_t^2 }{2x_t -1} \prod_{k \neq t}\frac{(x_k - x_t +1)(x_t + x_k)}{(x_k-x_t)(x_t + x_k - 1)} + \sum^r_{m=1} \frac{2 x_m(x_m -1)}{2x_m - 1} \prod_{k \neq m}\frac{(x_k - x_m +1)(x_k + x_m)}{(x_k-x_m)(x_k + x_m -1) } \\
			=&\sum^r_{t=1} 2x_t \prod_{k \neq t}\frac{(x_k - x_t +1)(x_t + x_k)}{(x_k-x_t)(x_t + x_k - 1)}.
		\end{align*}
		On cancelling 2 from both sides we get the desired identity.
	\end{proof}
	Next, we prove a ``Double Product Lemma'', involving the product of two sets of variables.
	\begin{lemma}[Double Product Lemma]\label{lem:doubleproductlemma}
		The identity
		\begin{align}\label{doubleproductlemma}
			&\frac{\prod_{k=1}^r x_k \prod_{k=1}^r y_k}{\prod_{k=1}^r (x_k+1) \prod_{k=1}^r (y_k+1)} \\
			=&1 -\sum_{l=1}^r \frac{1}{x_l+1} \prod_{k = 1}^r \frac{x_l+y_k+1}{x_l+y_k+2}\prod_{l \neq k =1}^r\frac{x_k-x_l-1}{x_k-x_l}- \sum_{l=1}^r \frac{1}{y_l+1}  \prod_{k = 1}^r \frac{y_l+x_k+1}{y_l+x_k+2}\prod_{l \neq k = 1}^r\frac{y_k-y_l-1}{y_k-y_l}
			\nonumber
		\end{align}
		holds for every $r \geq 1$.
	\end{lemma}
	\begin{proof}
		First, we multiply both sides of (\ref{doubleproductlemma}) by $\prod_{k=1}^r (x_k+1) \prod_{k=1}^r (y_k+1)$. Thus, we are proving the following identity:
		
		\[ \begin{split}
			\prod_{k=1}^r x_k \prod_{k=1}^r y_k =&\prod_{k=1}^r (x_k+1)(y_k+1)- 
			\sum_{l=1}^r \prod_{l\neq k=1}^r(x_k+1)\prod_{k=1}^r(y_k+1) \prod_{k = 1}^r \frac{x_l+y_k+1}{x_l+y_k+2}\prod_{l \neq k =1}^r\frac{x_k-x_l-1}{x_k-x_l}-\\
			&- \sum_{l=1}^r \prod_{l\neq k=1}^r(y_k+1)\prod_{k=1}^r(x_k+1)  \prod_{k = 1}^r \frac{y_l+x_k+1}{y_l+x_k+2}\prod_{l \neq k = 1}^r\frac{y_k-y_l-1}{y_k-y_l}.
		\end{split} \]
		We can put everything on the right-hand side to the common denominator to obtain 
		\[\mathrm{RHS}=\frac{A(x_1,\dots,x_r,y_1,\dots,y_r)}{\prod_{1\le k<l\le r}(x_k-x_l)(y_k-y_l)\prod_{1\le k,l\le r}(y_l+x_k+2)},\]
		where $A(x_1,\dots,x_r,y_1,\dots,y_r)=A(\overline x,\overline y)$ is a polynomial of degree at most $2r^2+r$.
		
		The next step is to see what happens when we exchange the values of $x_i$ and $x_j$. Clearly, the right-hand side does not change its value. However, the denominator in the equation above changes sign, thus also polynomial $A(\overline x,\overline y)$ must change sign. That means that $A(\overline x,\overline y)$ is divisible by $\prod_{1\le k<l\le r}(x_k-x_l)$ and after dividing we obtain a symmetric polynomial in $x_1,\dots,x_r$. 
		
		Analogously, the same holds when we exchange $y_i$ and $y_j$, and we can write \[A(\overline x,\overline y)=\prod_{1\le k<l\le r}(x_k-x_l)(y_k-y_l)\cdot B(\overline x,\overline y),\]
		where $B(\overline x,\overline y)$ is a polynomial of degree at most $r^2+2r$, symmetric in both $x_1,\dots x_r$ and $y_1,\dots, y_r$.
		
		Now we multiply the RHS by $(x_1+y_1+2)$ and plug in $x_1+y_1+2=0$ . Clearly all summands except those corresponding to $l=1$ vanish. Moreover for the two summands left we have
		
		\[\begin{split}
			&-(x_1+y_1+1)\prod_{ k=2}^r(x_k+1)\prod_{k=1}^r(y_k+1)\prod_{k =2}^r\frac{(x_1+y_k+1)(x_k-x_1-1)}{(x_1+y_k+2)(x_k-x_1)}-\\
			&-(x_1+y_1+1)\prod_{k=2}^r(y_k+1)\prod_{k=1}^r(x_k+1)\prod_{k = 2}^r\frac{(y_1+x_k+1)(y_k-y_1-1)}{(y_1+x_k+2)(y_k-y_1)}
		\end{split}\]
		
		After substituting $y_1=-2-x_1$, we obtain
		\[\begin{split}
			&-(-1)(-x_1-1)\prod_{k=2}^r(y_k+1)(x_k+1)\prod_{k =2}^r\frac{(x_1+y_k+1)(x_k-x_1-1)}{(x_1+y_k+2)(x_k-x_1)}-\\
			&-(-1)(x_1+1)\prod_{k=2}^r(x_k+1)(y_k+1)\prod_{k = 2}^r\frac{(x_k-x_1-1)(y_k+x_1+1)}{(x_k-x_1)(y_k+x_1+2)}=0.
		\end{split}\] 
		
		This implies that polynomial $B(\overline x,\overline y)$ must be divisible by $(x_1+y_1+2)$. From symmetry it is also divisible by $(x_i+y_j+2)$ for any $i,j$.
		
		Hence, \[B(\overline x,\overline y)=\prod_{1\le k,l\le r}(x_k+y_l+2)\cdot C(\overline x,\overline y),\]
		where $C(\overline x,\overline y)$ is a polynomial of degree at most $2r$, symmetric in $\overline x$ and $\overline y$. We then have that $\mathrm{RHS}=C(\overline x,\overline y)$.
		
		Next we plug in $x_1=-1$. Again, every term, except the one from the first sum for $l=1$ is 0. Thus, we get	
		\[\begin{split}
			C(-1,x_2,\ldots,x_r,y_1,\ldots,y_r)&=-\prod_{k=2}^r(x_k+1)\prod_{k=1}^r(y_k+1) \prod_{k = 1}^r \frac{x_1+y_k+1}{x_1+y_k+2}\prod_{k =2}^r\frac{x_k-x_1-1}{x_k-x_1} \\
			&=- \prod_{k=2}^r(x_k+1)\prod_{k=1}^r(y_k+1) \prod_{k = 1}^r \frac{y_k}{y_k+1}\prod_{k =2}^r\frac{x_k}{x_k+1}\\
			&=-\prod_{k=2}^r(x_k)\prod_{k=1}^r(y_k)=\prod_{k=1}^r(x_ky_k)
		\end{split}
		\]
		Therefore, for $x_1=-1$, the desired equality holds. Analogously, the same is true for any $x_i=-1$ or $y_i=-1$. This means that $C(\overline x,\overline y)-\prod_{k=1}^r(x_ky_k)$ is a polynomial divisible by $\prod_{k=1}^r(x_k+1)(y_k+1)$. 
		
		Furthermore, the only term of degree $2r$ in RHS is the first term $\prod_{k=1}^r(x_k+1)(y_k+1)$, as all other summands have degree at most $2r-1$. In particular, the degree $2r$ part of $C(\overline x,\overline y)$ is equal to $\prod_{k=1}^r(x_ky_k)$, and the difference $C(\overline x,\overline y)-\prod_{k=1}^r(x_ky_k)$ is of degree at most $2r-1$. Since it is divisible by a degree $2r$ polynomial, it must be 0, which proves the lemma.
	\end{proof}
	
	We conclude this section with a ``Product Lemma" derived from \Cref{lem:doubleproductlemma}, in a similar way as \Cref{cor:sumlemma} was obtained from \Cref{lem:doublesumlemma}.

	\begin{corollary}[Product Lemma]\label{cor:productlemma} For every positive integer $r$ the following identity holds:
		\[x_1\cdots x_r= \prod_{j = 1}^r(x_j+2)-2\sum_{l=1}^{r}\prod_{l\neq j=1}^r\frac{(x_j+2)(x_j-x_l-1)(x_j+x_l+2)}{(x_j-x_l)(x_j+x_l+3)}.  \]
	\end{corollary}
	\begin{proof}
		For every $1\leq i\leq r$, we specialize the identity in Lemma \ref{lem:doubleproductlemma} with $y_i=x_i+1$. The left hand side of the identity is then equal to $\frac{\prod_{k=1}^r x_k}{\prod_{k=1}^r (x_k+2)}$. For the right-hand side, we obtain
		\begin{align*}
			& 1 -
			\sum_{l = 1}^r \frac{1}{x_l+1} \prod_{k = 1}^r \frac{x_l+x_k+2}{x_l+x_k+3}\prod_{l \neq k = 1}^r\frac{x_k-x_l-1}{x_k-x_l}
			- \sum_{l = 1}^r \frac{1}{x_l+2}  \prod_{k = 1}^r \frac{x_l+x_k+2}{x_l+x_k+3}\prod_{l \neq k = 1}^r\frac{x_k-x_l-1}{x_k-x_l}\\=
			& 1 -
			\sum_{l = 1}^r \left(  \frac{1}{x_l+1}\cdot \frac{2x_l+2}{2x_l+3}+\frac{1}{x_l+2} \cdot\frac{2x_l+2}{2x_l+3}\right)\prod_{l \neq k = 1}^r\frac{(x_k-x_l-1)(x_l+x_k+2)}{(x_k-x_l)(x_l+x_k+3)}\\=
			& 1- \sum_{l = 1}^r \frac{2}{x_l+2}\prod_{l \neq k = 1}^r\frac{(x_k-x_l-1)(x_l+x_k+2)}{(x_k-x_l)(x_l+x_k+3)}.
		\end{align*}
		Multiplying both sides by $\prod_{j=1}(x_j+2)$ yields the desired identity.
	\end{proof}
	
	\section{Type C}\label{sec:C}
	
	The Lascoux polynomials play an essential role in proving the polynomiality of the ML-degree of linear concentration models. In this section we study the leading coefficient of these polynomials. Following  \cite{manivel2020complete}, we start by setting the notation and recalling the definition of Schur polynomials, Lascoux coefficients and Lascoux polynomials.

	A \emph{partition} $\lambda$ is a nonincreasing sequence of nonnegative integers $(\lambda_1,\dots,\lambda_r)$. The \emph{length} of the partition is the length of the sequence, the \emph{weight} is $\sum\lambda = \sum_{i=1}^r \lambda_i$.
	For a set $I = \{ i_1,\dots,i_r \}$ of nonnegative integers with $i_1 < i_2 < \dots < i_r$, we denote with $|I|$ its cardinality and with $\sum I =\sum_{j=1}^r i_j$. We associate to $I$ the corresponding partition
	\[ \lambda(I) = ( i_r-(r-1), i_{r-1}-(r-2),\dots,i_2-1,i_1 ). \]
	For a partition $\lambda$ of length $k$ its associated \emph{Schur polynomial} $s_\lambda$ is defined as follows:
	\[ s_\lambda(x_1,\dots,x_k) = \frac{\det (x_j^{\lambda_i + k - i})_{ij}}{\det (x_j^{k - i})_{ij}}. \]
	Note that the denominator of $s_\lambda$ is the Vandermonde determinant $\prod_{i < j}(x_i-x_j)$. The degree of $s_\lambda$ is equal to the weight $\sum\lambda$ of the partition. As an example, the elementary symmetric polynomial in $k$ variables of degree $r$ is the Schur polynomial with partition $\lambda = (\underbrace{1,\dots,1}_{r},0,\dots,0)$ of length $k$: 
	\[ s_{\lambda}(x_1,\dots,x_k) = \sum_{i_1 < \dots < i_r} x_{i_1} \dots x_{i_r}. \]
	Throughout the paper, the leading coefficient of a polynomial $p$ will be denoted by $LC(p)$.	
	
	% 	\harshit{I also think we should add a small remark somewhere saying briefly about the relation between Segre classes and the Lascoux coefficients to motivate algebraic geometric intuition about these things.}\lorenzo{Left for next meeting}
	\begin{definition}
		The \emph{Lascoux coefficients} are the numbers $\psi_I$ such that the following identity holds:
		\[
		s_{(d,0,\dots,0)}(\{x_i+x_j: 1\leq i\le j\leq k\}) = \sum_{\substack{|I|=k\\ \sum\lambda(I)=d }}\psi_I s_{\lambda(I)}(x_1,\dots,x_k).
		\]
		%For any increasing sequence $I=\{i_1,\dots,i_r\}$ of nonnegative integers, 
		The \emph{Lascoux polynomial} is the following function:
		\[
		\LP_I(n)=\begin{cases}
			\psi_{[n]\setminus I} & \text{ if }I\subseteq [n]\\
			0 & \text{otherwise}
		\end{cases}
		\]
	\end{definition}
	% 	\harshit{I think we should also mention in words "the coefficients appearing in the Schur expansion of the
	% complete homogeneous symmetric polynomial in sums of variables $x_i + x_j $" as done in notes of the workshop. This part can be slightly confusing for someone.}\lorenzo{Left for next meeting. $s_{(d)}$ $\to$ $s_{(d,0,\dots,0)}$ ?}
	
	\begin{remark}
		Notice that the previous definition differs from the one given in the introduction. In fact, there are many equivalent ways of defining the Lascoux coefficients. %An alternative definition was given in the introduction.
		While the one given in the introduction may be easier to state, the definition used here has the advantage of being naturally extended also for types A and D. However, in this article %we will not work with the definition, 
		we will just make use of the recurrence relations from \cite{manivel2020complete}, without worrying too much about which definition we use.  %Thus, we show the definition only for completeness. We chose this definition, since it can be naturally extended also for types A and D.
		There is also a geometrical way of defining Lascoux coefficients. More precisely, they are the Segre classes of the second symmetric power of the universal bundle over the Grassmannian.
		For more definitions and formulas about Lascoux coefficients we refer the reader to \cite[Appendix]{lascoux1989giambelli}.
	\end{remark}
	
	The authors in \cite{manivel2020complete} give different proofs that the Lascoux polynomials are indeed polynomials. The first is the simplest one, and is based on the following recurrence relations:
	
	Fix $I=\{i_1<i_2<\cdots<i_r\}\subset \mathbb{N}$.
	\begin{enumerate}
		\item If $i_1=0$, then 
		\begin{equation}\label{eq1}
			\LP_I(n)=(n-r+1)\LP_{I\backslash\{0\}}(n)-2\sum_{\substack{l>1\\i_{l+1}>i_l+1}}\LP_{I\cup \{i_l+1\}\backslash\{0,i_l\}}(n).
		\end{equation}
		
		\item If $i_1>0$, then
		\begin{equation}\label{eq2}
			\LP_I(n)-\LP_I(n-1)=\sum_{\epsilon\in\{0,1\}^r\backslash0}\LP_{I-\epsilon}(n-1),
		\end{equation}
		where $I-\epsilon:=\{i_1-\epsilon_1,\ldots,i_r-\epsilon_r\}$ and $\LP_{I-\epsilon}=0$ if there is a repeated element in $I-\epsilon$.
	\end{enumerate}
	
	The degree and leading coefficient of the Lascoux polynomials is also known:
	\begin{theorem} \cite[Theorem 4.12]{manivel2020complete}\label{thm:leading-coefficient-C}
		$\deg\LP_I=|I|+\sum I$, the leading coefficient is \[
		\frac{\prod_{j>k}(i_j-i_k)}{(i_1+1)!\cdots(i_r+1)!\prod_{j>k}(i_j+i_k+2)}.
		\]
	\end{theorem}
	
	However, the proof of this theorem in \cite{manivel2020complete} does not use the recurrence relations \eqref{eq1}, \eqref{eq2}, but a completely different approach, where the Lascoux coefficients $\psi_I$ are expressed as a sum of minors of the Pascal triangle matrix, using the definition given in the introduction. Here instead, we provide a direct proof of this theorem, by using just the recurrence relations \eqref{eq1}, \eqref{eq2}. Moreover, this method will be also useful in computing the degree and leading coefficient of Lascoux polynomials in types A and D, which is a new result in this article.
	
	\begin{proof}[Proof of Theorem \ref{thm:leading-coefficient-C}]
		We proceed analogously as in the first proof of polynomiality of Lascoux polynomials in \cite{manivel2020complete}. Thus, we proceed by induction, first on $|I|$, then on $\sum I$.
		
		The base case is $I=\emptyset$, when $\LP_I=1$ and the statement holds.
		
		For set $I$ we define $\lp_I$ to be the coefficient of $n^{|I|+\sum I}$ in $\LP_I(n)$. Next, we fix $I$ and assume that the statement is true for all $I'$ with $|I'|<|I|$ or $|I'|=|I|$ and $\sum I'<\sum I$. we consider two cases.
		
		\textbf{Case 1:} $i_1=0$. Then 
		\[\LP_I(n)=(n-r+1)\LP_{I\backslash\{0\}}(n)-2\sum_{\substack{l>1\\i_{l+1}>i_l+1}}\LP_{I\cup \{i_l+1\}\backslash\{0,i_l\}}(n)\]
		
		By the induction hypothesis, all terms on the right-hand side are polynomials of degree $|I|+\sum I$. Also from induction hypothesis we know their leading coefficients. Moreover, in the sum we can ignore the condition for $i_{l+1}>i_l$ simply by defining $\LP_{I'}:=0$, if $I'$ has repeated elements. Note that the formula for the leading coefficient holds in this case, since it is 0. Thus, by comparing the coefficients of $n^{|I|+\sum I}$ on both sides we get:
		\begin{align*}
			\lp_I=&\lp_{I\setminus\{0\}}-2\sum_{l>1}\lp_{I\cup \{i_l+1\}\backslash\{0,i_l\}}\\
			=&\frac{\prod_{j>k>1}(i_j-i_k)}{(i_2+1)!\cdots(i_r+1)!\prod_{j>k>1}(i_j+i_k+2)}-\\&-2\sum_{l=2}^{r}\frac{\prod_{j>k>1}(i_j-i_k)}{(i_2+1)!\cdots(i_r+1)!\prod_{j>k>1}(i_j+i_k+2)}\cdot\frac{1}{i_l+2}\cdot\prod_{l\neq j=2}^r\frac{(i_j-i_l-1)(i_j+i_l+2)}{(i_j-i_l)(i_j+i_l+3)}\\
			=&\frac{\prod_{j>k>1}(i_j-i_k)}{(i_2+1)!\cdots(i_r+1)!\prod_{j>k>1}(i_j+i_k+2)} \left(1-2\sum_{l=2}^{r} \frac{1}{i_l+2}\prod_{l\neq j=2}^r\frac{(i_j-i_l-1)(i_j+i_l+2)}{(i_j-i_l)(i_j+i_l+3)}\right)\\
			=&\frac{\prod_{j>k>1}(i_j-i_k)}{(i_2+1)!\cdots(i_r+1)!\prod_{j>k}(i_j+i_k+2)} \left(\prod_l (i_l+2)-2\sum_{l=2}^{r}\prod_{l\neq j=2}^r\frac{(i_l+2)(i_j-i_l-1)(i_j+i_l+2)}{(i_j-i_l)(i_j+i_l+3)}\right)\\
			=&\frac{\prod_{j>k>1}(i_j-i_k)}{(i_2+1)!\cdots(i_r+1)!\prod_{j>k}(i_j+i_k+2)}\cdot i_2\cdots i_k =\frac{\prod_{j>k}(i_j-i_k)}{(i_2+1)!\cdots(i_r+1)!\prod_{j>k}(i_j+i_k+2)},
		\end{align*}
		where we applied Corollary \ref{cor:productlemma}	for $i_2,\dots,i_r$ with $r=k$, $x_1=0$ and $x_j=i_j$ for $2\leq j\leq k$. This proves the theorem in this case.
		
		\textbf{Case 2:} $i_1>0$. Then
		\[\LP_I(n)-\LP_I(n-1)=\sum_{\epsilon\in\{0,1\}^r\backslash0}\LP_{I-\epsilon}(n-1)\]
		By induction hypothesis, all terms on the right-hand side are polynomials of degree $|I|+\sum I-\sum_{i=1}^r \epsilon_i$ with positive leading coefficients. Thus, the right-hand side is a polynomial of degree $|I|+\sum I-1$, and to the coefficient of $n^{|I|+\sum I-1}$ contribute only terms for $\sum_{i=1}^r \epsilon_i = 1$.
		
		It follows that $\LP_I$ is a polynomial of degree $|I|+\sum I$ and the coefficient of $n^{|I|+\sum I-1}$ is $(|I|+\sum I)\lp_I$. Using the induction hypothesis we can compare the leading coefficients of both sides and get:
		\begin{align*}
			(i_1+\dots+i_r+r)\lp_I=&\sum_{l=1}^r\frac{\prod_{j>k}(i_j-i_k)}{(i_1+1)!\cdots(i_r+1)!\prod_{j>k}(i_j+i_k+2)}(i_l+1)\prod_{l\neq j=1}^{r}\frac{(i_j-i_l+1)(i_j+i_l+2)}{(i_j-i_l)(i_j+i_l+1)}\\
			=&\frac{\prod_{j>k}(i_j-i_k)}{(i_1+1)!\cdots(i_r+1)!\prod_{j>k}(i_j+i_k+2)} \left(\sum_{l=1}^{r}(i_l+1) \prod_{l\neq j=1}^r\frac{(i_j-i_l+1)(i_j+i_l+2)}{(i_j-i_l)(i_j+i_l+1)}\right)\\
			=&\frac{\prod_{j>k}(i_j-i_k)}{(i_1+1)!\cdots(i_r+1)!\prod_{j>k}(i_j+i_k+2)} \cdot (i_1+\dots+i_r+r),
		\end{align*}
		where we used Corollary \ref{cor:sumlemma} with $x_j=i_j+1$ for $1\leq j\leq r$. Then the statement follows by cancelling $(i_1+\dots+i_r)$ from both sides.
	\end{proof}
	
	\section{Type A}\label{sec:A}
	% The Type A Lascoux polynomials were introduced in \cite[Theorem 6.11] {manivel2020complete}.
	In \cite[Section 6]{manivel2020complete}, the authors have defined the Type A Lascoux functions. They have also proved that these functions are indeed polynomials in $n$. The aim of this section is to find a formula for the leading coefficient and the degree of these polynomials using the recurrence relations given in \cite[Lemma 6.10, Theorem 6.11]{manivel2020complete}. The proof is very similar to the one for type C given in Section \ref{sec:C}.
	
	% We denote $D(t)$ to be the infinite matrix with entries $D(t)_{i,j} = \binom{t+i+j}{i}$. 
	% \begin{definition}
	% Let $I = \{i_1,\dots,i_r\}, J = \{j_1,\dots,j_s\}$ with $r\leq s$ then we define 
	% \[
	% d_{I,J} = \begin{cases}
	% \det D(s-r)_{I,\{j_{s-r+1}-(s-r)}
	% \end{cases}
	% \] 
	% \end{definition}

	\begin{definition}
		\begin{comment}
		We define \chen{the \emph{Lascoux coefficients of type A}} $d_{I,J}$ as follows. For $X = (x_1,\dots,x_k)$ and $Y= (y_1,\dots,y_l)$ two sets of indeterminates, we denote by $X + Y$ the set of indeterminates $x_i + y_j, 1 \leq i \leq k, 1 \leq j \leq l$. Then the $d_{I,J}$ is defined by the following identity
		
		\[
		s_{(d,0,\dots,0)}(X+Y) = \sum_{\substack{\chen{|I|} =k, \chen{|J|} =l \\ |\lambda(I)| + |\lambda(J)| = d}} d_{I,J} s_{\lambda(I)}(X)s_{\lambda(J)}(Y).
		\]
		\end{comment}
		For $X = (x_1,\dots,x_k)$ and $Y= (y_1,\dots,y_l)$ two sets of indeterminates, denote by $X + Y$ the set of indeterminates $\{x_i + y_j, 1 \leq i \leq k, 1 \leq j \leq l\}$. The \emph{Lascoux coefficients of type A} are the numbers $d_{I,J}$ such that the following identity holds:
		\[
		s_{(d,0,\dots,0)}(X+Y) = \sum_{\substack{|I| =k, |J| =l \\ \sum \lambda(I) + \sum \lambda(J)= d } } d_{I,J} s_{\lambda(I)}(X)s_{\lambda(J)}(Y).
		\]
		
		The \emph{Lascoux polynomials of type A} are given by
		\[
		\LP_{I,J}^A(n)=\begin{cases}
			d_{[n]\setminus I,[n]\setminus J} & \text{ if }I,J\subseteq [n]\\
			0 & \text{otherwise.}
		\end{cases}
		\]
	\end{definition}
	%	\subsection{Recurrence relations}\label{recu_A}
	
	The Lascoux polynomials of type A are indeed polynomial functions in
	$n$ \cite[Theorem 6.11]{manivel2020complete}.	These polynomials satisfy the following two recurrence relations \cite[Lemma 6.10, Theorem 6.11]{manivel2020complete}. Fix $I=\{i_1<i_2<\cdots<i_r\}\subset \mathbb{N}$ and $J=\{j_1<j_2<\cdots<j_r\}\subset \mathbb{N}$.
	\begin{enumerate}
		\item  If $i_1 = 0$ and $j_1 = 0$, then
		\begin{equation}\label{recu_A_1}
			\begin{split}
				\LP_{I,J}^A(n) & = (n-r+1)
				\LP_{I \setminus \{0\}, J \setminus \{0\}}(n) -\\
				& - \sum_{\ell: i_{\ell+1} > i_\ell+1} \LP_{I \setminus \{0,i_\ell\} \cup \{i_{\ell}+1\}, J \setminus \{0\}}^A(n)
				- \sum_{\ell: j_{\ell+1} > j_\ell+1} \LP_{I \setminus \{0\}, J \setminus \{0, j_{\ell}\}\cup \{j_{\ell}+1\}}^A(n).
			\end{split}
		\end{equation}
		
		\item Otherwise, if $i_1 > 0$ or $j_1 > 0$,
		\begin{equation}\label{recu_A_2}
			\LP^A_{I,J}(n) = \sum_{I',J'} \LP^A_{I',J'}(n-1)
		\end{equation}
		where the sum is over all pairs $(I', J')$ of the form $(\{i_1 - \epsilon_1, \dots , i_r - \epsilon_r\},\{j_1 - \mu_1, \dots,j_r - \mu_r\} ) $, where $\epsilon_l,\mu_l \in \{0,1\}$.  
	\end{enumerate}

	% 	\begin{lemma} The degree of the leading term of Type A Lascoux polynomials is given by the following expression on $I, J$.
	% 	    \[
	% 	\deg(\LP^A_{I,J}(n)) = |I| + \chen{\sum} I + \chen{\sum} J
	% 	\]
	% 	\end{lemma}
	% 	\begin{proof}
	% 		By Problem 6.1.5 in the notes.\\ 
	% 		\textcolor{red}{need to add a small proof here}
	% 	\end{proof}
	\begin{remark}
		The degree of the Lascoux polynomials of type A satisfies the following inequality:
		\[
		\deg(\LP^A_{I,J}(n)) \leq |I| + \sum I + \sum J
		\]
	\end{remark}
	\begin{theorem}\label{thm: leadingcoefA}
		For sets $I=\{i_1,...,i_r\}$, $J=\{j_1,...,j_r\}$, 
		the degree of the Lascoux polynomials of type A is given by the following expression on $I, J$:
		\[
		\deg(\LP^A_{I,J}(n)) = |I| + \sum I + \sum J
		\]
		and the leading coefficient of $\LP^A_{I,J}$ is
		\[\frac{\prod_{k>l}(i_k-i_l)\prod_{k > l} (j_k-j_l)}{\prod_{k,l=1}^r(i_k+j_l+1) \prod_{k=1}^r (i_k)! \prod_{k=1}^r (j_k)!}.\]
	\end{theorem}
	\begin{proof}
		We will proceed by induction, first on $|I|$ and then on $\sum I + \sum J$. The proof is analogous to the proof of Theorem \ref{thm:leading-coefficient-C} in Type C case. 
		We denote $\ell_{I,J}$ to be the coefficient of $n^{|I| +\sum I + \sum} J$ in $\LP_{I,J}(n)$. As $\LP_{I,J}$ have two recurrence relations given in \eqref{recu_A_1} and \eqref{recu_A_2}, we will get corresponding recurrence relations for $\ell_{I,J}$.
		
		\textbf{First recursion:\\} From the first recursion \eqref{recu_A_1} by comparing the coefficients of degree $|I| + \sum I + \sum J$ we get
		
		\[ \ell_{I,J} = \ell_{I_0, J_0} - \sum_{\substack{l > 1 \\ i_{l+1} > i_l+1}} \ell_{I_l, J_0} - \sum_{\substack{l > 1 \\ j_{l+1} > j_l+1}} \ell_{I_0, J_l} \]
		where $I_0 = I \setminus \{ 0 \}$, $J_0 = J \setminus \{ 0 \}$, $I_l = I \cup \{ i_l+1 \} \setminus \{0,i_l \}$ and $J_l = J \cup \{ j_l+1 \} \setminus \{0,j_l \}$. Now write
		\begin{align*}
			\ell_{I_0, J_0} & = \frac{\prod_{k>t>1}(i_k-i_t)\prod_{k > t > 1} (j_k-j_t)}{\prod_{k,t=2}^r(i_k+j_t+1) \prod_{k=2}^r (i_k)! \prod_{k=2}^r (j_k)!} \\
			\ell_{I_l, J_0} & = \frac{\prod_{k>t>1}(i_k-i_t)\prod_{k > t > 1} (j_k-j_t)}{\prod_{k,t=2}^r(i_k+j_t+1) \prod_{k=2}^r (i_k)! \prod_{k=2}^r (j_k)!} \cdot \frac{1}{i_l+1} \prod_{k = 2}^r\frac{i_l+j_k+1}{i_l+j_k+2}\prod_{l \neq k = 2}^r\frac{i_k-i_l-1}{i_k-i_l} \\
			\ell_{I_0, J_l} & = \frac{\prod_{k>t>1}(i_k-i_t)\prod_{k > t > 1} (j_k-j_t)}{\prod_{k,t=2}^r(i_k+j_t+1) \prod_{k=2}^r (i_k)! \prod_{k=2}^r (j_k)!} \cdot \frac{1}{j_l+1} \prod_{k = 2}^r\frac{j_l+i_k+1}{j_l+i_k+2}\prod_{l \neq k = 2}^r\frac{j_k-j_l-1}{j_k-j_l}.
		\end{align*}
		Note that
		\[ \ell_{I,J} = \frac{\prod_{k=1}^r i_k \prod_{k=1}^r j_k}{\prod_{k=1}^r (i_k+1) \prod_{k=1}^r (j_k+1)} \cdot \ell_{I_0, J_0}. \]
		Now write
		\[ \ell_{I,J} = \ell_{I_0, J_0} \left( 1 -
		\sum_{l > 1} \frac{1}{i_l+1} \prod_{k = 2}^r\frac{i_l+j_k+1}{i_l+j_k+2}\prod_{l \neq k = 2}^r\frac{i_k-i_l-1}{i_k-i_l}
		- \sum_{l > 1} \frac{1}{j_l+1} \prod_{k = 2}^r\frac{j_l+i_k+1}{j_l+i_k+2}\prod_{l \neq k = 2}^r\frac{j_k-j_l-1}{j_k-j_l}
		\right) \]
		and apply Lemma \ref{lem:doubleproductlemma} with $x_k=i_k$ and $y_k=j_k$ for every $1\leq k\leq r$.
		
		\textbf{Second recursion:\\}
		From the second recursion (\ref{recu_A_2}) by comparing the coefficients of degree $|I| + \sum I + \sum J - 1$ we get
		\[ \deg \LP_{I,J} \cdot \ell_{I,J} = \sum_{t = 1}^r \ell_{I_t,J} + \sum_{t = 1}^r \ell_{I,J_t} \]
		where $I_t = \{ i_1,\dots,i_t-1,\dots,i_r\}$ and $J_t = \{ j_1,\dots,j_t-1,\dots,j_r\}$. Now write
		\[ \ell_{I_t,J} =  \frac{\prod_{k > l}(i_k-i_l)\prod_{k > l} (j_k-j_l)}{\prod_{k,l=1}^r(i_k+j_l+1) \prod_{k=1}^r (i_k)! \prod_{k=1}^r (j_k)!} i_t \prod_{k \neq t}\frac{i_k - i_t +1}{i_k-i_t} \prod_{l=1}^r\frac{i_t+j_l+1}{i_t+j_l}, \]
		
		\[ \ell_{I,J_t} =  \frac{\prod_{k > l}(i_k-i_l)\prod_{k > l} (j_k-j_l)}{\prod_{k,l=1}^r(i_k+j_l+1) \prod_{k=1}^r (i_k)! \prod_{k=1}^r (j_k)!} j_t \prod_{k \neq t}\frac{j_k - j_t +1}{j_k-j_t}\prod_{l=1}^r\frac{i_l+j_t+1}{i_l+j_t}. \]
		Therefore,
		\[
		\begin{split}
			\deg(\LP^A_{I,J}(n)) = & \sum^r_{t=1} i_t \prod_{k \neq t}\frac{i_k - i_t +1}{i_k-i_t} \prod_{l=1}^r\frac{i_t+j_l+1}{i_t+j_l} + \sum^r_{m=1} j_m \prod_{k \neq m}\frac{j_k - j_m +1}{j_k-j_m}\prod_{l=1}^r\frac{i_l+j_m+1}{i_l+j_m},
		\end{split}
		\]
		which is equal to $|I|+\sum I+\sum J$ by Lemma~\ref{lem:doublesumlemma} with $x_k=i_k$ and $y_k=j_k$ for every $1\leq k\leq r$.
		%\[ \ell_{I_t,J} = i_t\dfrac{\prod_{k>t}(i_k - i_t+1) \prod_{t>k}(i_t-i_k-1)\prod_{k>l, k,l\neq t}(i_k - i_l)\prod_{k>l}(j_k - j_l)}{\prod_{l}(i_t+j_l)\prod_{ \substack{k\neq t, \\ l}}(i_k+j_l+1) \prod_{l}(i_l)!\prod_{l}(j_l)!} \]
		
	\end{proof}

	\section{Type D}\label{sec:D}
	
	In this section, we turn our attention to the type D case, and proceed in a way analogous to the previous sections. The Lascoux functions for type D were first defined in \cite[Section 7]{manivel2020complete}. Here we provide a formula for their degree and their leading coefficients.
	
	\begin{definition}
		The \emph{Lascoux coefficients of type D} are the numbers $\alpha_I$ which verify the identity
		\[
		s_{(d,0,\dots,0)}(\{x_i+x_j: 1\leq i < j\leq n\}) = \sum_{\substack{|I|=n\\ \sum \lambda(I) =d }}\alpha_I s_{\lambda(I)}(x_1,\dots,x_n).
		\]
		For any increasing sequence $I=\{i_1,\dots,i_s\}$ of nonnegative integers the \emph{Lascoux quasipolynomial of type D} is
		\[
		\LP^D_I(n)=\begin{cases}
			\alpha_{[n]\setminus I} & I\subseteq [n],\\
			0 & \text{otherwise}
		\end{cases}
		\]
	\end{definition}
	
	In \cite[Theorem 7.10]{manivel2020complete} it was proved that $\LP_I^D(n)$ is a quasipolynomial of period $2$, in other words, $\LP^D_I(2n)$ and $\LP^D_I(2n-1)$ are polynomials in $n$. The proof of this result uses the following recursive relations. Fix
	$I=\{i_1<i_2<\cdots<i_r\}\subset \mathbb{N}$.
	\begin{enumerate}
		\item If $i_1=0$, then 
		
		\begin{equation}\label{recD1}
			\LP^D_I(n)=\begin{cases} \LP^D_{I\setminus\{0\}}(n) &\text{ if } n-|I| \text{ is even, }\\
				0 &\text{ if } n-|I| \text{ is odd. } 
			\end{cases}
		\end{equation}
		
		\item If $i_1>0$, then
		\begin{equation}\label{recD2}
			\LP^D_I(n)-\LP^D_I(n-1)=\sum_{\epsilon\in\{0,1\}^n\backslash0}\LP^D_{I-\epsilon}(n-1),
		\end{equation}
		where $I-\epsilon:=\{i_1-\epsilon_1,\ldots,i_r-\epsilon_r\}$ and $\LP^D_{I-\epsilon}=0$ if $|I-\epsilon|<r$. In the main result of this section we compute degree and leading coefficient of the quasipolynomials $\LP^D_I(n)$.
	\end{enumerate}
	
	%\begin{theorem}\label{leadingcoefD}
	%	Let $I=\{i_1<\cdots <i_r\}\subset \mathbb{N}$ be a set of nonnegative integers. If $i_1>0$, the function $\LP^D_{I}:\mathbb{N}\rightarrow \mathbb{R}$, $n\mapsto \LP^D_{I}(n)$ is a quasipolynomial in $n$ with period 2 and degree $\sum I$ and constant leading coefficient 
	%	\[
	%	\frac{\prod_{k>l}(i_k-i_l)}{2^{|I|}\prod_{k>l}(i_k+i_l) \prod_k (i_k)!}.
	%	\]
	%	If $i_1=0$, then $\LP_I^D(n)$ is a quasipolynomial as well and is equal to $\LP_{I\setminus\{0\}}^D(n)$ if $n-|I|$ is even and 0 if $n-|I|$ is odd.
	%\end{theorem}
	
	\begin{theorem}\label{thm: leadingcoefD}
		Let $I=\{i_1<\cdots <i_r\}\subset \mathbb{N}$ be a set of nonnegative integers. Then
		\begin{itemize}
			\item[-] If $i_1>0$, $\LP^D_{I}(2n)$ and $\LP^D_{I}(2n+1)$ are polynomials in $n$ of degree $\sum I$ and leading coefficient equal to
			\[
			\frac{2^{\sum I-|I|}\prod_{k>l}(i_k-i_l)}{\prod_{k>l}(i_k+i_l) \prod_k (i_k)!}.
			\]
			\item[-] If $i_1=0$, then $\LP^D_{I}(n)=\LP^D_{I\setminus\{0\}}(n)$ if $n-|I|$ is even, and  $\LP^D_{I}(n)=0$ if $n-|I|$ is odd.  
		\end{itemize}
		
	\end{theorem}

	\begin{proof}
		We fix a set $I=\{i_1<\cdots <i_r\}\subset \mathbb{N}$. For the case $i_1=0$, the statement follows from induction hypothesis. 
		
		In the case $i_1=0$, the statement is nothing but the recurrence relation \eqref{recD1}. Now we consider the case $i_1>0$ and proceed by induction.
		
		Assume that the statement holds for all $I'$ with $|I'|<|I|$ or $|I'|=|I|$ and $\sum I'<\sum I$. By applying (\ref{recD2}) we obtain 
		\begin{equation}\label{eq: type D}
			\LP^D_I(2n)-\LP^D_I(2(n-1))=\sum_{\epsilon\in\{0,1\}^n\backslash0}\LP^D_{I-\epsilon}(2n-1)+\sum_{\epsilon\in\{0,1\}^n\backslash0}\LP^D_{I-\epsilon}(2n-2).
		\end{equation}
		
		On the right-hand side we get sum of polynomials of degree at most $\sum I-1$. Moreover, the polynomials with this degree are only those with $\epsilon_1+\dots+\epsilon_r=1$. Let $\ell^D_{I}$ be the coefficient of $n^{\sum I}$ in $\LP_I^D(2n)$, and let $e_i\in \{0,1\}^n$ be the $i$-th vector of the canonical basis of $\mathbb{Z}^n$. Comparing the coefficients of $n^{\sum I - 1}$ of both sides of \eqref{eq: type D} and using the induction hypothesis on $\ell^D_{I-e_j}$ we obtain
		\begin{align*}
			\deg(\LP^D_I(2n))\ell^D_I=&2\sum_{j=1}^r \ell^D_{I-e_j}\\
			=&2\sum_{j=1}^{r}\left(\frac{\prod_{k>l}(i_k-i_l)2^{\sum I-|I|-1}}{\prod_{k>l}(i_k+i_l) \prod_k (i_k)!}\cdot i_j \prod_{k \neq j} \frac{(i_k+i_j)(i_k-i_j+1)}{(i_k-i_j)(i_k+i_j-1)}\right)\\
			=&\frac{\prod_{k>l}(i_k-i_l)2^{\sum I-|I|}}{\prod_{k>l}(i_k+i_l) \prod_k (i_k)!}\left(\sum_{j=1}^{r}i_j \prod_{k \neq j} \frac{(i_k+i_j)(i_k-i_j+1)}{(i_k-i_j)(i_k+i_j-1)}\right).
		\end{align*}
		Notice that if $i_1=1$, one of $\LP_{\{0,i_2,\ldots,i_r\}}(2n-1)$ and $\LP_{\{0,i_2,\ldots,i_r\}}(2n-2)$ is zero and the other is equal to $\LP^D_{\{i_2,\ldots,i_r\}}(2n-1)$ (or $\LP^D_{\{i_2,\ldots,i_r\}}(2n-2)$) whose leading coefficient is also the same as the term included in the expression above. By \Cref{cor:sumlemma}, the last expression is equal to
		\[
		\frac{\prod_{k>l}(i_k-i_l)2^{\sum I-|I|}}{\prod_{k>l}(i_k+i_l) \prod_k (i_k)!}\left(\sum_{j=1}^{r}i_j \right).
		\]
		As this quantity is not zero for any set $I$, we have that $\ell^D_{I}\neq 0$, and hence $\deg(\LP^D_I(2n))=\sum I$. It follows that $n\mapsto \LP^D_{I}(2n)$ is a polynomial function in $n$ with degree $\sum I$ and leading coefficient 
		\[\frac{\prod_{k>l}(i_k-i_l)2^{\sum I-|I|}}{\prod_{k>l}(i_k+i_l) \prod_k (i_k)!}.\]
		The case of the polynomial function for $n\mapsto \LP^D_{I}(2n-1)$ is completely analog. This concludes the proof.
		
	\end{proof}
	\section{The algebraic degrees $\delta(m,n,n-1)$, $\delta_A(m,n,n-1)$ and $\delta_D(m,n,n-1)$}\label{sec:delta}
	
	One of the applications of the results in \cite{manivel2020complete} establishes polynomiality of a sequence of positive integers attached to \emph{semidefinite programming}. This is the problem of optimizing a linear function over the cone of positive semidefinite matrices. In \cite{NRS} the authors study the complexity of computing an exact solution for this optimization problem, and they quantify this complexity via the degree of a projective variety. Similar degrees can be defined for optimization problems related to the space of general and skew-symmetric matrices. We recall that for a variety $X\subseteq \mathbb{P}^n$ the \emph{projective dual} $X^*\subseteq (\mathbb{P}^n)^*$ is the closure of the set of hyperplanes tangent to $X$ at a smooth point. In the next definition we follow the notation in \cite{manivel2020complete}.
	\begin{definition}[{\cite[Definition 1.4, 6.2 and 7.2]{manivel2020complete}}]
		We define the following three numbers:	
		\begin{itemize}
			\item[Type C] Let $SD^{r,n}_m\subseteq \mathbb{P}(S^2 \mathbb{C}^n)$ be the intersection of the variety of $n\times n$ symmetric matrices of rank at most $r$ with a general linear space of projective dimension $m$. We define $\delta(m,n,r)$ as the degree of $(SD^{r,n}_m)^*$ if it is a hypersurface, and zero otherwise.
			\item[Type A] Let $D^{r,n}_m\subseteq \mathbb{P}(\mathbb{C}^n\otimes \mathbb{C}^n)$ be the intersection of the variety of $n\times n$ matrices of rank at most $r$ with a general linear space of projective dimension $m$. We define $\delta_A(m,n,r)$ as the degree of $(D^{r,n}_m)^*$ if it is a hypersurface, and zero otherwise.
			\item[Type D] Let $AD^{2r,2n}_m\subseteq \mathbb{P}(\bigwedge^2\mathbb{C}^n)$ be the intersection of the variety of $2n\times 2n$ skew-symmetric matrices of rank at most $2r$ with a general linear space of projective dimension $m$. We define $\delta_D(m,n,r)$ as the degree of $(AD^{2r,2n}_m)^*$ if it is a hypersurface, and zero otherwise.
		\end{itemize}
	\end{definition}
	Using Lascoux polynomials, in \cite{manivel2020complete} it is proved that $\delta(m,n,n-s)$, $\delta_A(m,n,n-s)$ and $\delta_D(m,n,n-s)$ are polynomials in $n$. We determine their degrees using our results on the leading coefficient of Lascoux polynomials. Moreover, we compute the leading coefficients of $\delta(m,n,n-s)$, $\delta_A(m,n,n-s)$ and $\delta_D(m,n,n-s)$ in the case \alessio{when} $s=1$ combining the formulas obtained in the previous sections together with the results in \cite{manivel2020complete}. These results should be regarded as asymptotic degrees.
	
	Here and in the rest of the section we denote with $LC(f)$ the leading coefficient of a univariate polynomial $f$.
	\begin{theorem}[Type C]\label{thm: semidefinite type C} For every $s>0$ and $m\geq \binom{s+1}{2}$, the polynomial $\delta(m,n,n-s)$ has degree $m$. Moreover
		\[
		LC(\delta(m,n,n-1)) = \frac{2^{m-1}}{m!},
		\]			
		for every $m>0$.
	\end{theorem}
	\begin{proof}
		By \cite[Theorem 1.1]{BR} we have that
		\[
		\delta(m,n,n-s) = \sum_{\substack{I\subseteq [n]\\ |I| = s\\ \sum I = m-s} } \psi_I \LP_I(n),
		\]
		where $\psi_{I}$ are the Lascoux coefficients as in \cite[Definition 2.5]{manivel2020complete}. Observe that the last sum is not empty if and only if there exists $I\subseteq [n]$ with $\sum I \geq \binom{|I|}{2}$. This happens if and only if $m\geq \binom{s+1}{2}$. By \Cref{thm:leading-coefficient-C} for every fixed $m,s$ satisfying this inequality, $\delta(m,n,n-s)$ is a positive finite linear combination of polynomials of degree $|I| +\sum I = m$, which proves the first claim.
		For the second statement we have $\delta(m,n,n-1) =\psi_{\{m-1\}} \LP_{\{m-1\}}(n)$ and 
		$LC(\delta(m,n,n-1)) =\psi_{\{m-1\}} LC(\LP_{\{m-1\}}(n))$.
		In \cite[Lemma 2.7]{manivel2020complete} it is proved that $\psi_{\{m-1\}}=2^{m-1}$, and by \Cref{thm:leading-coefficient-C} we have that $LC(\LP_{\{m-1\}}(n))=\frac{1}{m!}$. This concludes the proof.
	\end{proof}
	Hence, for large values of $n$ we have that $\delta(m,n,n-1)\sim \frac{2^{m-1}}{m!} n^{m}$.
	
	\begin{theorem}[Type A]\label{thm: semidefinite type A} For every $s>0$ and $m\geq s^2$ the polynomial $\delta_A(m,n,n-s)$ has degree $m$. Moreover,
		\[
		LC(\delta_A(m,n,n-1)) = \frac{1}{m!}\binom{2(m-1)}{m-1}.
		\]			
	\end{theorem}
	\begin{proof}
		By \cite[Theorem 6.8]{manivel2020complete} we have that
		\[
		\delta_A(m,n,n-s)=\sum_{\substack{I,J\subset [n]\\ |I| = |J| = s\\ \sum I + \sum J=m-s}} d_{I,J} \LP^A_{I,J}(n),
		\]
		where $d_{I,J}$ are the type A Lascoux coefficients as defined in \cite[Definition 6.7]{manivel2020complete}. The last sum is not empty if and only if the condition $\sum I+\sum J\geq \binom{|I|}{2}+\binom{|J|}{2}$ is satisfied by some $I,J\subseteq [n]$. This is equivalent to $m-s\geq 2\binom{s}{2}$, that is $m\geq s^2$. Hence by Theorem~\ref{thm: leadingcoefA} for fixed $m,s$ satisfying this inequality,  $\delta_A(m,n,n-s)$ is a positive finite combination of polynomials of degree $|I| +\sum I + \sum J = m$.
		For the second statement we specialize to $r=n-1$ and obtain $\delta_A(m,n,n-1)=\sum_{i=0}^{m-1} d_{\{i\},\{m-1-i\}} \LP^A_{\{i\},\{m-1-i\}}(n)$. As by Theorem~\ref{thm: leadingcoefA} all the $m$ polynomials on the right-hand side have the same degree we have that 
		\[
		LC(\delta_A(m,n,n-1))=\sum_{i=0}^{m-1} d_{\{i\},\{m-1-i\}} LC(\LP^A_{\{i\},\{m-1-i\}}(n)).
		\]
		By \cite[Proposition 6.9]{manivel2020complete} we have that $d_{\{i\},\{m-1-i\}}=\binom{m-1}{i}$, and by Theorem~\ref{thm: leadingcoefA} we have that $LC(\LP^A_{\{i\},\{m-1-i\}}(n))=\frac{1}{m\cdot i!(m-1-i)!}$. Combining the two results we obtain
		\begin{align*}
			LC(\delta_A(m,n,n-1))&=\sum_{i=0}^{m-1}\binom{m-1}{i}\frac{1}{m\cdot i!(m-1-i)!}=\frac{1}{m}\sum_{i=0}^{m-1}\binom{m-1}{i}^2=\frac{1}{m!}\binom{2(m-1)}{m-1}.	
		\end{align*}
	\end{proof}
	\Cref{thm: semidefinite type A} implies that for large values of $n$, $\delta_A(m,n,n-1)\sim \frac{1}{m!}\binom{2(m-1)}{m-1} n^{m}$.\\
	
	Finally, we present analog results for the type D case. 
	
	\begin{theorem}[Type D]\label{thm: semidefinite type D} For every $s>0$ and $m\geq \binom{2s}{2}$, the polynomial $\delta_D(m,n,n-s)$ has degree $m$. Moreover, 
		\[
		LC(\delta_D(m,n,n-1)) =\frac{2^{m-2}}{m!}\left(\frac{1}{m}\binom{2(m-1)}{m-1}+1\right).
		\]			
	\end{theorem}
	\begin{proof}
		
		In \cite[Theorem 7.8]{manivel2020complete} it is proved that
		\[
		\delta_D(m,n,r)=\sum_{\substack{I\subset [2n]\\ |I| =2n-2r\\ \sum I=m}}\alpha_I \LP^D_{I}(2n).
		\]
		The sum on the right-hand side is not empty if and only if $\sum I\geq \binom{|I|}{2}$, that is $m\geq \binom{2s}{2}$. When this inequality holds,
		by \Cref{thm: leadingcoefD} we have that $\delta_D(m,n,n-s)$ is a positive finite combination of polynomials of degree $\sum I = m$.
		Moreover, $\delta_D(m,n,n-1)=\sum_{i=0}^{\lfloor\frac{m-1}{2}\rfloor} \alpha_{\{i,m-i\}} \LP^D_{\{i,m-i\}}(2n)$ and we obtain 
		\[
		LC(\delta_D(m,n,n-1))=\sum_{i=0}^{\lfloor\frac{m-1}{2}\rfloor} \alpha_{\{i,m-i\}} LC(\LP^D_{\{i,m-i\}}(2n)).
		\]
		By \cite[A.16.5]{lascoux1989giambelli} we have that $\alpha_{\{i,j\}}=\binom{i+j-1}{i}-\binom{i+j-1}{i-1}$, so in particular $\alpha_{i,m-i}=\binom{m-1}{i}-\binom{m-1}{i-1}$. Using \Cref{thm: leadingcoefD} we conclude that 
		\[
		LC(\LP^D_{\{i,m-i\}}(2n))=\begin{cases}
			\frac{m-2i}{4m\cdot i!(m-i)!} & i>0\\
			\frac{m-2i}{2m\cdot i!(m-i)!} & i=0
		\end{cases}.
		\] 
		As a polynomial in $n$, the degree of $\LP^D_{\{i,m-i\}}(2n)$ is equal to $m$ and its leading coefficient is then $2^mLC(\LP^D_{\{i,m-i\}}(2n))$. We obtain
		\begin{align*}
			LC(\delta_D(m,n,n-1))&=\frac{2^{m-1}}{ m!}+\sum_{i=1}^{\lfloor\frac{m-1}{2}\rfloor}\left(\binom{m-1}{i}-\binom{m-1}{i-1} \right)\frac{2^{m-2}(m-2i)}{m\cdot i!(m-i)!}\\
			&=\frac{2^{m-1}}{ m!}+\frac{2^{m-2}}{m^2\cdot m!}\sum_{i=1}^{\lfloor\frac{m-1}{2}\rfloor}(m-2i)^2\binom{m}{i}^2\\
			&=\frac{2^{m-1}}{ m!}+\frac{2^{m-2}}{m^2\cdot m!}\left( m\binom{2(m-1)}{m-1}-m^2\right)\\
			&= \frac{2^{m-2}}{m!}\left(\frac{1}{m}\binom{2(m-1)}{m-1}+1\right).
		\end{align*}
		
	\end{proof}	
	Hence, for large values of $n$ we have that $\delta_D(m,n,n-1)\sim \frac{2^{m-2}}{m!}\left(\frac{1}{m}\binom{2(m-1)}{m-1}+1\right) n^{m}$.\\
	It is of course possible to follow the same idea to compute the leading coefficients of $\delta(m,n,n-s)$, $\delta_A(m,n,n-s)$ and $\delta_D(m,n,n-s)$ for higher values of $s$, even though the calculation becomes significantly more involved. We conclude this article with a natural question.
	\begin{problem}
		Find formulas in $m$ and $s$ for the leading coefficients of $\delta(m,n,n-s)$, $\delta_A(m,n,n-s)$ and $\delta_D(m,n,n-s)$.
	\end{problem}

	\bibliographystyle{alpha} 
	\bibliography{references.bib}
	
\end{document}